\documentclass[11pt]{amsart}
\usepackage{graphicx}

\newtheorem{thm}{Theorem}[section]

\newtheorem{lem}[thm]{Lemma}
\newtheorem{prop}[thm]{Proposition}
\theoremstyle{definition}

\newtheorem{ass}[thm]{Assumption}

\theoremstyle{remark}
\newtheorem{rem}[thm]{Remark}

\numberwithin{equation}{section}

\newcommand{\abs}[1]{\left\vert#1\right\vert}
\newcommand{\set}[1]{\left\{#1\right\}}
\newcommand{\Real}{\mathbb R}
\newcommand{\Natural}{\mathbb N}

\newcommand{\such}{\ | \ }
\newcommand{\prob}{\mathbb{P}}

\newcommand{\qprob}{\mathbb{Q}}
\newcommand{\expec}{\mathbb{E}}

\newcommand{\F}{\mathcal{F}}

\newcommand{\ud}{\mathrm d}
\newcommand{\inner}[2]{\left \langle #1 , #2 \right \rangle}

\newcommand{\nin}{n \in \Natural}
\newcommand{\limt}{\lim_{t \to \infty}}

\newcommand{\e}{\mathrm{e}}

\newcommand{\pare}[1]{\left(#1\right)}
\newcommand{\bra}[1]{\left[#1\right]}
\newcommand{\dbra}[1]{[\kern-0.15em[ #1 ]\kern-0.15em]}
\newcommand{\dbraco}[1]{[\kern-0.15em[ #1 [\kern-0.15em[}

\newcommand{\indic}{\mathbb{I}}

\newcommand{\probbb}{\prob_{\mathsf{BB}^3}}
\newcommand{\covbb}{\mathrm{cov}_{\mathsf{BB}^3}}

\newcommand{\expecbb}{\expec_{\mathsf{BB}^3}}

\newcommand{\probax}{\prob^{a}_x}
\newcommand{\pax}{p^{a}_x}
\newcommand{\px}{p_x}
\newcommand{\lgx}{{\lambda_x}}
\newcommand{\phx}{\widehat{p}_x^N}
\newcommand{\betah}{\widehat{\beta}}

\newcommand{\dfn}{\, := \,}

\begin{document}

\title[Monte Carlo Estimation of diffusion first passage time densities]{Efficient estimation of one-dimensional diffusion first passage time densities via Monte Carlo simulation}%

\author{Tomoyuki Ichiba}%
\address[Tomoyuki Ichiba]{Department of Statistics and Applied Probability, University of California Santa Barbara, CA 93106, South Hall 5607A}
\email{ichiba@pstat.ucsb.edu}

\author{Constantinos Kardaras}%
\address[Constantinos Kardaras]{Department of Mathematics and Statistics, Boston University, 111 Cummington Street, Boston, MA 02215, USA}
\email{kardaras@bu.edu}

\thanks{The second author gratefully acknowledges partial support by the National Science Foundation, grant number DMS-0908461}%

\date{\today}%
\begin{abstract}
We propose a method for estimating first passage time densities of one-dimensional diffusions via Monte Carlo simulation. Our approach involves a representation of the first passage time density as expectation of a functional of the three-dimensional Brownian bridge. As the latter process can be simulated exactly, our method leads to almost unbiased estimators. Furthermore, since the density is estimated directly, a convergence of order $1 / \sqrt{N}$, where $N$ is the sample size, is achieved, the last being in sharp contrast to the slower non-parametric rates achieved by kernel smoothing of cumulative distribution functions.
\end{abstract}

\maketitle

\setcounter{section}{-1}

\section{Introduction}

The problem of computing the distribution of the first time that a diffusion crosses a certain level naturally arises in many different contexts. As probably the most prevalent we mention quantitative finance, where first passage times are used in credit risk (times of default) as well as in defining exotic contingent claims (so-called barrier options). In this paper, we focus on the numerical computation of the probability density function of first-passage times associated with a general one-dimensional diffusion.

Densities of first passage times have known analytic expressions only in very particular cases. The primary example is Brownian motion with certain (constant) drift and diffusion rates, where one uses a combination of Girsanov's theorem and the special case of standard (driftless) Brownian motion. The first passage time density for the latter case can be obtained by the reflexion principle --- see, for example, \cite[2.6A]{MR1121940}. 
Another example of explicit form of the first-passage time density is the case of the radial Ornstein-Uhlenbeck process --- see \cite{MR1725406} and \cite{GJY03}.

In absence of general analytic expressions for first passage time distributions, computational methods are indispensable and, in fact, widely used. One branch uses Volterra integral equations --- we mention \cite{MR776891}, \cite{MR1984752} and \cite{MR2591904} as representative papers dealing with this approach. Alternatively, one can use Monte-Carlo simulation to attack this problem. The simplest scheme uses the so-called Euler scheme in order to approximate the solution of the stochastic differential equation governing the diffusion at predetermined grid time-points $i h$, $i = 0, 1, \ldots$, where $h$ is the step size, stops at the first time that the diffusion crosses the level of interest, and continuing this way obtain an estimator for the cumulative distribution function of the first-passage time.  As the Euler scheme is only approximate\footnote{To account for the fact the passage can potentially happen in-between the sampled points, a Brownian bridge interpolation may be used --- this means that the conditional probability distribution of first-passage time occurred in the interval given the simulated values at the end points is approximated by that of a Brownian bridge. This could potentially reduce the bias, but the whole scheme is still only approximate and some bias remains.}, it causes bias in the estimation of the probability distribution function. (See \cite{MR1999614} for general discussion, \cite{MR1757112} for the evaluation of error via partial differential equations and \cite{MR1849251} for a sharp large deviation principle approach.) The issue with the bias becomes immensely more severe in the numerical computation of the density, as it will involve some kind of numerical differentiation of the (non-smooth) empirical distribution function. Even if one uses an exact simulation approach for the diffusion in question (which is, of course, available only in special cases), the estimator for the density will have huge variance. To top it all, even if all the aforementioned problems can be eliminated, one can never hope for convergence of the estimators to the true density in order $1/ \sqrt{N}$, where $N$ is the ``path-sample'' size, as the problem is non-parametric. 

\smallskip

In this work, we offer an alternative approach which has clear advantages. First, we arrive at a representation of the density function in terms of expectation of a functional of a three-dimensional Brownian bridge. This makes it possible to estimate \emph{directly} the first passage time density without having to rely on estimators of the cumulative distribution function, achieving this way the ``parametric'' rate of convergence $1 / \sqrt{N}$, where $N$ is the sample size. Furthermore, only the three-dimensional Brownian bridge is involved in the simulation, which can be carried out exactly. There is an integral involving the previous three-dimensional Brownian bridge, which can be approximated via a Riemann sum; therefore, the error of the approximation can be estimated efficiently. By construction, our method significantly improves both the bias and variance of the density estimation obtained via the empirical distribution function. The only potential problem of our approach is large-time density estimation, since the thin grid that has to be used in the simulation of the Brownian bridge will result in high computational effort. To circumvent this issue, we notice that the tails of the first-passage distribution usually decrease exponentially with a rate that can be expressed as the principal eigenvalue of a certain Dirichlet boundary problem involving a second-order ordinary differential equation. This implies that a mixture of Monte-Carlo and ordinary differential equation techniques can be efficiently utilized and improve the quality of our estimator.


The structure of the paper is simple. In Section 1, the problem is formulated and the key representation formula is obtained. In Section 2, we discuss the Monte-Carlo estimator of the first passage time density function, and study its large-sample properties. In Section 3, the relation between the exponential tail decay of the probability density and the eigenvalues of a Dirichlet boundary problem is discussed. The proofs of all the results are deferred to Appendix \ref{sec: proofs} in order to keep the presentation smooth in the main body of the paper.

\section{A Representation of First-Passage-Time Densities}

\subsection{The set-up} Consider a one-dimensional diffusion $X$ with dynamics
\begin{equation} \label{eq: diffusion}
\ud X_t = a (X_t) \ud t + \ud W_t, \quad t \in \Real_+.
\end{equation}
Above, $W$ is a standard one-dimensional Brownian motion. The restrictions on the drift function $a$ that we shall impose later (Assumptions \ref{ass: basic}) ensure that 
(\ref{eq: diffusion}) has a weak solution, unique in the sense of probability law, for any initial condition $x \in (0, \infty)$.  Let $\probax$ will denote the law on the canonical path-space $C (\Real_+; \Real)$ of continuous functions from $\Real_+$ to $\Real$ that makes the coordinate processes behave according to \eqref{eq: original diffusion} and is such that $\probax [X_0 = x] = 1$.

Define $\tau_0 \dfn \inf \set{t \in \Real_+ \such X_t = 0}$ to be the first passage time of $X$ at level zero. 
We shall consider the problem of finding convenient, in terms of numerical approximation using the Monte-Carlo simulation technique, representations of the quantity
\[
\pax (t) \dfn \frac{\partial}{\partial t}  \probax \bra{
\tau_0 \leq t } ; \quad x \in (0, \infty)\, , \, \, t \in \Real_+,
\]
i.e., the density of the first-passage time of the diffusion at level zero.

\begin{rem}
The fact that we are using unit diffusion coefficient in \eqref{eq: diffusion} by no means entails loss of generality in our discussion. Indeed, consider a general one-dimensional diffusion $Y$ with dynamics
\begin{equation} \label{eq: original diffusion}
\ud Y_t = b (Y_t) \ud t + \sigma(Y_t) \ud W_t, \quad t \in \Real_+,
\end{equation}
where $W$ is a standard one-dimensional Brownian motion, such that $Y_0 = y \in \Real$. If \eqref{eq: original diffusion} has a weak solution unique in the sense of probability law, we may assume without loss of generality that $\sigma \geq 0$. (Indeed, otherwise we replace $\sigma$ by $|\sigma|$ in \eqref{eq: original diffusion} and we obtain the same law for the process $Y$.) Consider a level $\ell < y$. Under the mild assumption that $1 / \sigma$ is locally integrable, the transformation $X = \int_\ell^{Y} \pare{1 / \sigma(z)}\ud z$ defines a diffusion with dynamics $\ud X_t = a(X_t) \ud t + \ud W_t$, for a function $a$ that is easily computable from $b$ and $\sigma$. With $x \dfn \int_\ell^y (1 / \sigma(z)) \ud z$, the first passage time of $Y$ with $Y_0 = y$ at level $\ell$ is equal to the first passage time of $X$ with $X_0 = x$ at level zero.
\end{rem}

\subsection{The representation}

The following assumption on the drift function in \eqref{eq: diffusion} will allow us to arrive at a very convenient representation for the density function $\pax$.
\begin{ass} \label{ass: basic}
The function $a$ restricted on $[0, \infty)$ is continuously differentiable, and satisfies 
\[
\int^\infty_0 \exp \pare{ - 2 \int^w_0 a(z) \ud z } \ud w = \infty.
\]
\end{ass}
In particular, under Assumption \ref{ass: basic}, $a$ is locally square integrable on $[0, \infty)$ and the function
\begin{equation} \label{eq: gamma}
\gamma  \dfn   \frac{a^2 + a'}{2}.
\end{equation}
is continuous and locally integrable.  The theory of one-dimensional diffusions ensures that for all $x \in (0, \infty)$ there exists a probability $\probax$ on $C(\Real_+, \Real)$ such that the coordinate process $X$ has dynamics given by \eqref{eq: diffusion}. Assumption \ref{ass: basic} also ensures that $X^{\tau_0}$, which is $X$
stopped at level zero, does not explode to infinity --- see, for example, \cite[Proposition 5.32 (iii)]{MR1121940}.

\begin{prop} \label{prop: p with BB3}
Suppose that Assumption \ref{ass: basic} is in force. On $C([0,1]; \, \Real^3)$, consider the probability $\probbb$ under which the coordinate process $\beta$ is a standard 3-dimensional Brownian bridge. Then,
\[
\pax(t)  =    q_{x} (t) \exp \left( - \int_0^x a(v) \ud v \right)  \expecbb \left[ \exp \left(  - t\int_0^1 \gamma \left( \left|  u x \e_1 + \sqrt{t} \beta_u \right|
\right) \ud u \right) \right]
\]
holds for all $t \in \Real_+$, where $\e_1 \dfn (1, 0, 0)^\top$ and $q_{x}$ is the density given for all $t \in \Real_+$ by
\begin{equation} \label{eq: dens 0}
q_x(t) \equiv p^{0}_{x} (t) = \frac{x}{\sqrt{2 \pi t^3}} \exp \pare{- \frac{x^2}{2 t}},
\end{equation}
corresponding to the first-passage time to zero of a standard Brownian motion starting from $x$.
\end{prop}

\section{Monte-Carlo Density Estimation}

We now discuss issues related to estimation of the density $\pax$. For the purposes of this and the next section, we \emph{fix} a drift function $a$ satisfying Assumption \ref{ass: basic} and we write $\px$ for $\pax$ in order to simplify notation.

\subsection{Convergence}
It is clear how to get an estimate of the density $\px (t)$ for a given $t \in \Real_+$, at least in theory. One simulates $N$ independent paths of 3-dimensional Brownian bridge $\betah^1, \ldots, \betah^N$, and then defines the estimator $\phx (t)$ for $\px(t)$ via
\begin{equation} \label{eq: phx} 
\phx (t) \dfn q_x (t) \exp \left( - \int_0^x a(v) \ud v \right) 
\frac{1}{N} \sum_{i = 1}^N \exp \left(  - t\int_0^1 \gamma \left( \left|  u x \e_1 + \sqrt{t} \betah^i_u \right|
\right) \ud u \right),
\end{equation}
where recall that $q_x$ is given in \eqref{eq: dens 0}. By the strong law of large numbers, the estimator $\phx (t)$ converges almost surely to the true density $\px (t)$ as $N$ goes to infinity for each fixed $t \in \Real_+$. Moreover, the estimator $ \phx(t) $ is unbiased\footnote{Of course, $ \expecbb \bra{\phx(t)} = \px(t)$ only holds if we assume that we actually have the whole path of each Brownian bridge simulated exactly, which is not possible in practice. However, we can simulate exactly discretized paths of the Brownian bridge, and then can easily estimate the order of bias from the Riemann approximation of the integral. In this respect, see also \S \ref{subsec: practical issues}.} and  the variance of the estimator $ \phx(t) $ decreases in the order of $ 1 / N$, for every fixed $ t \in \Real_{+} $, as a direct consequence of (\ref{eq: phx}). 
In order to get weak convergence of the whole empirical densities $\pare{\phx(t)}_{t \in \Real_+}$, as well as the uniform rate of convergence over compact time-intervals, we introduce an additional assumption.

\begin{ass} \label{ass: extra}
Together with Assumption \ref{ass: basic}, we ask that the function $\gamma$ of \eqref{eq: gamma} is such that $\inf_{z \in \Real_+} \gamma(z) > -\infty$, as well as that $a^{\prime}$ is locally Lipschitz continuous on $\Real_+$ with Lipschitz constant growing at most polynomial rate, that is, there exist constants $ c_1 > 0$, $ c_2 > 0$, 
$n \in \Natural$, such that
\begin{equation} \label{eq: poly growth}
\sup_{0 \le v_1 < v_2 \le \kappa} \abs{ \frac{a^\prime (v_2 ) - a^\prime(v_1)}{v_2 - v_1} }
\le (c_1 + c_2 \kappa^n)  \text{ holds for all } \kappa > 0. 
\end{equation}
\end{ass}

The next two results are concerned with a central limit theorem for the whole density function estimator as well as the uniform rate of convergence on compact intervals of $ \Real_+$.

\begin{prop}\label{prop: tight}
Suppose that Assumption \ref{ass: extra} holds. For $N \in \Natural$, define $\eta^{N} \dfn \sqrt{N} \pare{\phx - \px} $. Then, the family of stochastic processes $\set{ \eta^{N} \such N \in \Natural}$ is tight.
As $N\to \infty$, $\eta^N$ converges weakly to a centered Gaussian process with continuous covariance function $\Gamma$, where, with $I(t) =
\int_0^1 \gamma \left( \left|  u x \e_1 + \sqrt{t} \beta_u \right|
\right) \ud u$ for $t \in \Real_+$, 
\[
\Gamma (s, t) = p_x(s)p_x(t) \exp \Big( -2 \int^x_0 a(v) \ud v\Big)  \covbb [\exp (-sI(s)), \exp (-t I(t))],  \quad (s,t) \in \Real_+^2.  
\]
\end{prop}

\begin{prop}\label{prop: conv in prob}
Under Assumption \ref{ass: extra}, for any fixed $T \in \Real_+$ the sequence
\[
\pare{\sqrt{N}  \max_{t \in [0, T]} \abs{\phx (t)
- \px(t)}}_{N \in \Natural}
\]
 is bounded in probability.
\end{prop}

\begin{rem}
In a similar manner, for fixed $x_1$ and $x_2$ in $(0, \infty)$ with $x_1 < x_2$, we may show that
\[
\pare{\sqrt{N} \max_{x \in [x_1, x_2]} \abs{\phx (t) - \px(t)}}_{N \in \Natural} 
\]
is bounded in probability. Moreover, under some additional conditions on differentiability of $a$, we may estimate the partial  derivatives of $\px(t)$  with respect to $(x, t)$ via differentiating the estimator with respect to the variable of interest. 
\end{rem}

\section{The Rate Function}

Recall that we are dropping the qualifying ``$a$'' from ``$\pax$'' in order to simplify notation. Define implicitly the function $\lgx$ via
\[
\px (t) = q_{x} (t) 
 \exp \pare{ -  \int^{x}_{0} a(v) \ud v } \exp \pare{ - t \lgx (t)}, \quad t \in \Real_+.
\]
In other words, and in view of Proposition \ref{prop: p with BB3}, we have
\begin{equation} \label{eq: lgx}
\lgx (t)  \dfn - \frac{1}{t} \log \left( \expecbb \left[ \exp \left(  - t\int_0^1 \gamma \left( \left|  u x \e_1 + \sqrt{t} \beta_u \right|
\right) \ud u \right) \right] \right), \text{ for } t \in \Real_+.
\end{equation}

\subsection{Theoretical results}

Proposition \ref{prop: conv in prob} ensures the uniform convergence of the estimator on finite intervals $[0, T]$ for fixed $T \in (0, \infty) $. Of course, it is almost always the case that $\lim_{t \to \infty} \px(t) = 0$. 
For large $t \in \Real_+$, $\lgx(t)$ gives a better understanding of the behavior of the density function, as it represents in a certain sense the exponential decrease of $\px(t)$; therefore, it makes more sense to focus on $\lgx$ rather than $\px$. In fact, the following result implies that the function $\lgx$ is frequently bounded on $\Real_+$ --- this is the case, for example, when $\gamma$ is bounded from below.

\begin{prop}\label{lm: bounds lambda}
Let Assumption \ref{ass: basic} be valid. Then,
\begin{equation} \label{eq: estimate lgx limit}
\inf_{z\in \Real_+} \gamma (z) \le \inf_{t \in \Real_+}\lgx (t)
\le \limsup_{t\to \infty} \lgx (t)
\le \inf_{\kappa>0} \left\{ m(\kappa  +x ) -
\frac{\pi^2}{2\kappa^2} \right\}, 
\end{equation}
where $m(w) = \max_{0\le z \le w} \gamma (z) $ for $w > 0$. Furthermore,
if $\gamma$ is bounded from below, then
\begin{equation} \label{eq: lgx at zero}
\lim_{t\downarrow 0} \lgx (t) = \frac{1}{x} \int^{x}_{0} \gamma (u) \ud u
= \int^{1}_{0} \gamma (u x) \ud u.
\end{equation}
\end{prop}

\begin{rem}\label{rem: lgx limit} 
The inequalities (\ref{eq: estimate lgx limit}) only imply bounds for the inferior and superior limit of $\lgx (t)$ as $t \to \infty$. In fact, it is expected that $\limt \lgx(t)$ exists, possibly except in pathological cases. Let us argue below for this point on a rather loose and intuitive level.

The stopping time $\tau_0$ can be approximated by the sequence  $(\tau_0^n)_{\nin}$, where, for all $\nin$, $\tau_0^n \dfn \inf \{ t \ge 0 \such X_t \not \in (0, n) \}$  is the first exit time of $X$ from the 
interval $(0,n)$. Therefore, the rate function $\lgx$ may be approximated by the corresponding rate functions corresponding to $\tau_0^n$, $\nin$. It follows from \cite{MR0068701} that the density function 
$p_{x}^{n}$ of $\tau_0^n$ has the eigenvalue 
expansion (see also \cite{MR576891}, \cite{MR1326606} for similar  problems) 
\[
p_{x}^{n}(t) = \frac{\partial}{\partial t} \, \probax \bra{
\tau_0^n \leq t } = \sum_{k=1}^\infty 
e^{-\mu^n_k \, t} \, \psi^n_k (x), \text{ for } (t, x) \in \Real_+ \times (0, n), \ \nin,  
\] 
for some functions $\{ \psi^n_k \such k \in \Natural \}$ 
computed from the eigenfunctions $\{ \varphi^n_k \such k \in \Natural \}$ and the corresponding eigenvalues $ 0 < \mu^n_1 < \mu^n_2 < \ldots$ of the Dirichlet problem 
\begin{equation} \label{eq: eigenvalue prob}
\frac{1}{2} \varphi''(z) + a(z) \varphi'(z) = - \mu  \,   
\varphi(z)  \, \quad 
\text{ for } z \in (0, n), 
\end{equation}
for $\varphi \in C^{2, \alpha}([0,n], \Real_+)$ with 
$\lim_{z \to 0} \varphi(z)= 0 = 
\lim_{z \to n} \varphi(z)$. Thus, the limit as $t \to \infty$ for the rate function of $\tau_0^n$ is exactly the principal eigenvalue 
$\mu^n_1$: 
\[
\lim_{t\to \infty} \lambda_{x}^{n} (t) = - \lim_{t\to \infty} \frac{1}{t} \log \Big( \frac{p^{n}_{x}(t)}{q_x(t)}\Big) = \mu^n_1 \text{ for all } 
\nin,  
\]
which does not depend on the initial value $x > 0$. Since $\tau_0 = \lim_{n \to \infty} \tau_{0,n}$, it is conjectured that the limit of $\lambda_{x} (t)$ as $t \to \infty$ actually exists and is equal to $\lim_{n \to \infty} \mu^n_1$. A thorough study of finding a reasonable sufficient conditions for
\[
\lim_{t \to \infty} \lambda_x (t) = \lim_{n \to \infty} \mu^n_1 = \lim_{n \to \infty} \lim_{t\to \infty} \lambda_{x}^{n}(t)
\]
to hold for $x > 0$ lies beyond the scope of this paper. 
\end{rem}

\subsection{Practical issues} \label{subsec: practical issues}
In view of Proposition \ref{prop: tight} and Proposition \ref{prop: conv in prob}, the estimator $ (\phx(t))_{t \in \Real_{+}} $ of (\ref{eq: phx}) convergences uniformly with rate $1 / \sqrt{N}$ over compact time-intervals. In practice, the computation of (\ref{eq: phx}) is implemented by generating a standard 3-dimensional Brownian bridge, which is simulated in an exact way over a thin enough grid. The approximation error for the Riemann integral over the {\it finite} interval $[0, 1]$ in (\ref{eq: phx}) can be controlled very efficiently. More precisely, the numerical computation of the exponential functional of the Brownian bridge in (\ref{eq: phx}) can be carried out using the fourth-order Runge-Kutta scheme which is proposed and analyzed in \cite{MR2061247}. Under appropriate mild regularity conditions on the function $\gamma$, it is shown that this numerical scheme is weak order of convergence $4$. For this numerical issue, consult the original paper \cite{MR2061247} and the related monographs \cite{MR1999614}, \cite{MR1214374} and \cite{MR2069903}.

A potential problem with our estimator $ (\phx(t))_{t \in \Real_{+}} $ can arise for {\it large} $ t $, that is, the density function at the tail. Note that what is meant here is that the \emph{relative} error of the estimator of $\px(t)$ tends to be large; the absolute error tends to be extremely small, as $\px(t)$ is very close to being zero for large $t \in \Real_+$. To visualize the issue, it is helpful to study by experiment the large-time behavior of the rate function (\ref{eq: lgx}) when $X$ is an Ornstein-Uhlenbeck (OU) process starting with $x = 1$. Here  $a(z) =  -z$ and $ \gamma(z) = (z^2-1) / 2$ for $z \in  \Real_+$. The first-passage time density of this OU process is known analytically (see, for example, \cite[equation (8)]{GJY03}) and reads 
\begin{equation} \label{eq: OU density}
p_1 (t) = \frac{1}{\sqrt{2\pi}} \frac{1}{\sinh^{3/2} (t)} \exp 
\Big( \frac{1+t-\coth(t)}{2} \Big), \text{ for } t \in \Real_+. 
\end{equation}
\begin{figure} 
\begin{center}
\begin{tabular}{cc}
\includegraphics[scale=0.4]{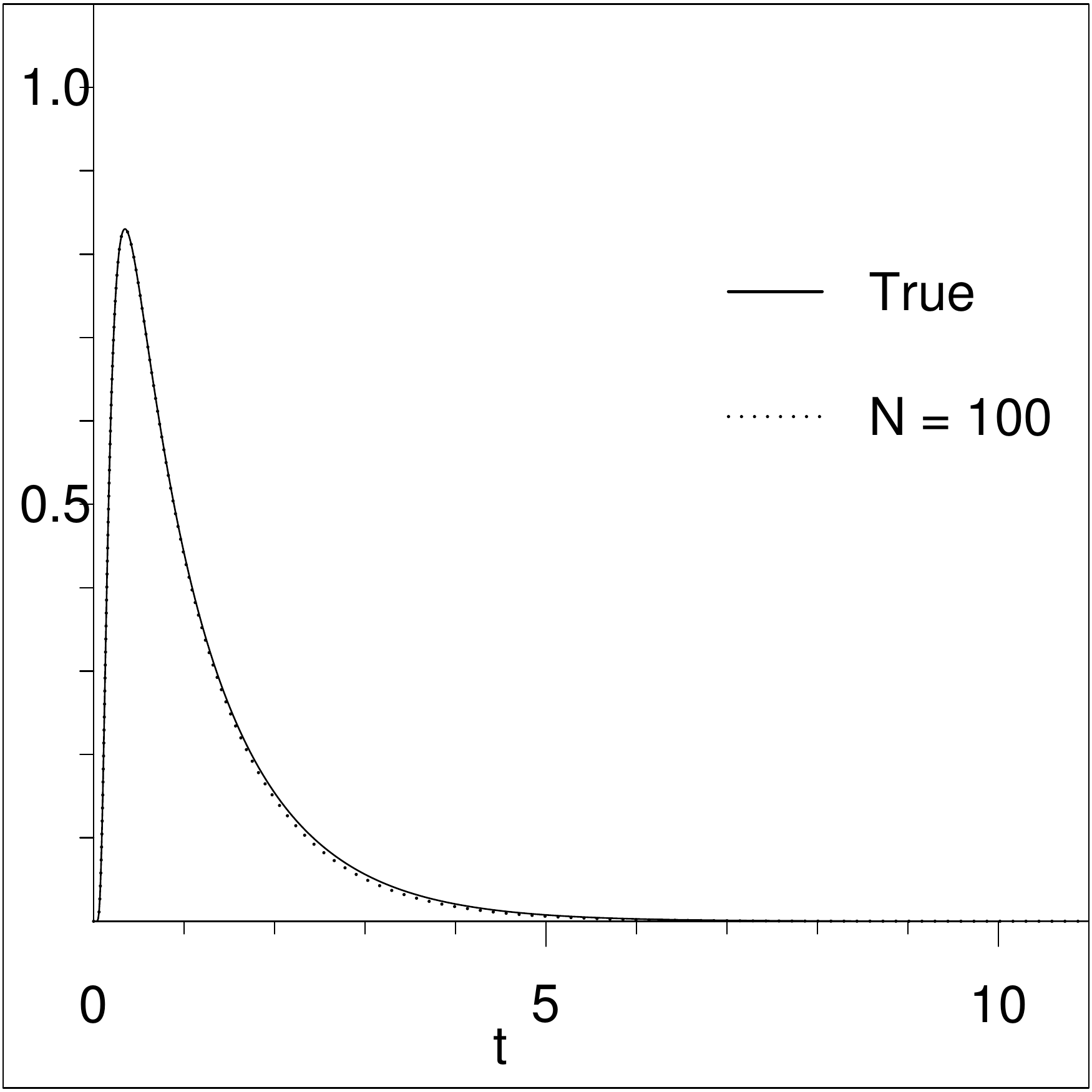}
&
\includegraphics[scale=0.4]{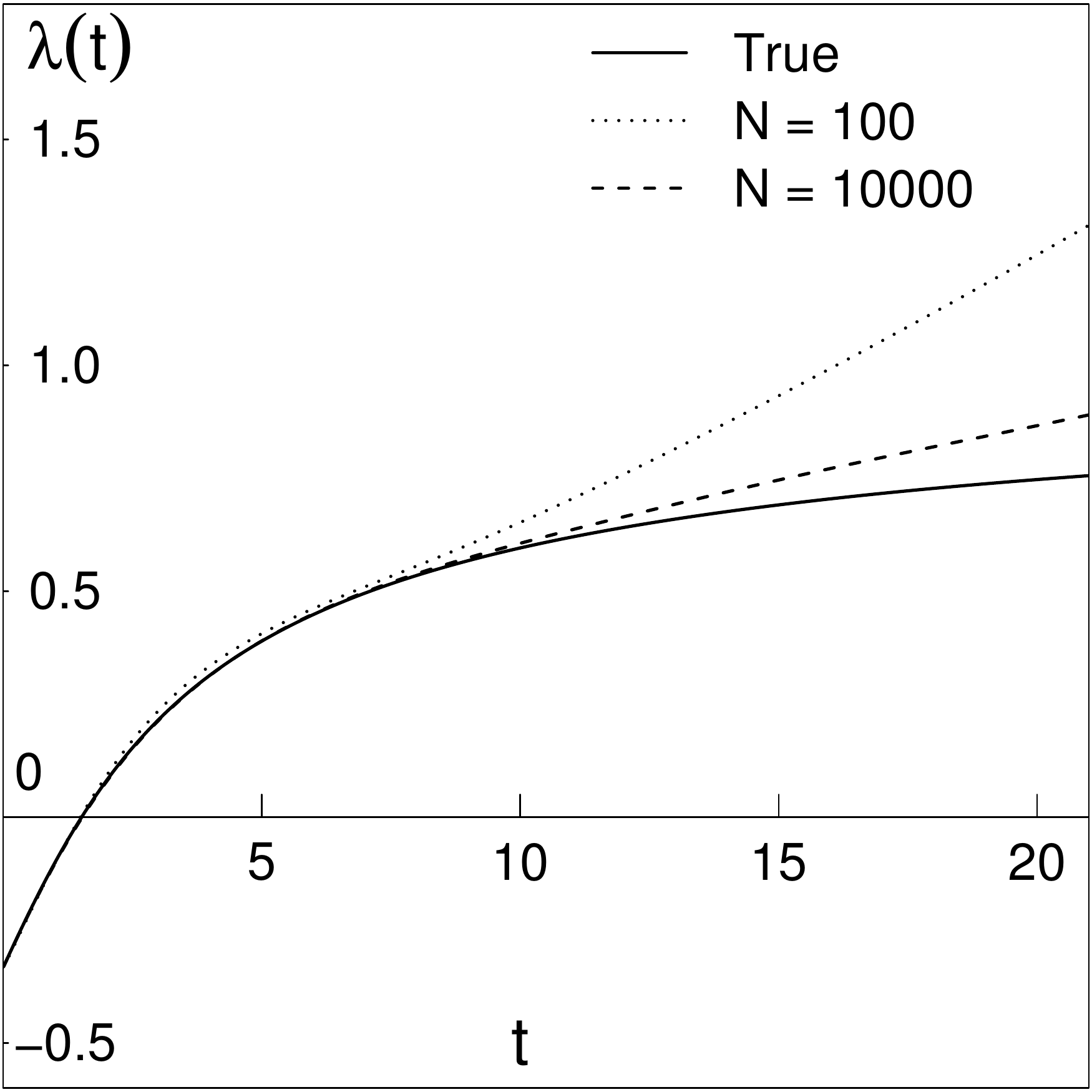} \\
(a) Density functions $ \px$ 
and $ \phx$. & 
(b) Rate functions  $ \lgx $ and 
$ \widehat{\lambda}^{N}_{x}$. \\
\end{tabular}
\end{center}
\caption{Comparisons of the simulated 
density and rate functions 
in (\ref{eq: phx}) and (\ref{eq: l est})
with theoretical values for 
the OU process 
with $ a(z) = -z$ ($z \in \Real_+$) and $ x = 1$.}
\label{fig: 1}
\end{figure}

The true density (\ref{eq: OU density}) and the estimated density (\ref{eq: phx}) with $N= 100$ simulations are shown over interval $ [0, 10]$ 
in Figure \ref{fig: 1}(a). Note that even with this small number of simulations, the two curves are almost indistinguishable. Figure \ref{fig: 1}(b) contains graphs of the more refined rate functions. In this scale, one can see that the estimation of the exponential rate of decay of the density for large $t$ is not as good. (As can be seen from \ref{fig: 1}(b), on the interval $ [0,7] $, the true and estimated rate functions almost coincide. However, on the interval $ [7, 20]$, the estimated rate functions $ \widehat{\lambda}^N_x$ are much larger than the true rate function $ \lgx$, which implies that the estimator (\ref{eq: phx}) underestimates the tail probability.) In fact, the estimated asymptotic rate seems to increase linearly instead of converging to a finite limit.  This becomes clear once one notes that, in this OU example,
\[
- t \int_0^1 \gamma \left( \left|  u x \e_1 + \sqrt{t} \beta_u \right|
\right) \ud u = t^2\frac{1}{2} \int_0^1 |\beta_u|^2 \ud u + t^{3/2} x \int_0^1 \inner{\e_1}{\beta_u} \ud u + t \frac{x^2}{6};
\]
therefore, the estimator for $\lgx$ for sample size $N \in \Natural$ becomes
\begin{equation} \label{eq: l est}
\widehat{\lambda}^{N}_{x} (t)  = \frac{x^2}{6} - \frac{1}{t} \log \left( 
\frac{1}{N} \sum^{N}_{i=1}  \exp \left(  t^2\frac{1}{2} \int_0^1 |\widehat{\beta}^i_u|^2 \ud u + t^{3/2} x \int_0^1 \inner{\e_1}{\widehat{\beta}^i_u} \ud u  \right) \right).
\end{equation}
Observe that the leading term in the estimator of $\lgx$ will be increasing linearly in $t$.
 
In order to overcome this poor situation in the tail, one can use a mixture method combining the estimator (\ref{eq: phx}) on the finite interval $[0, T]$ and an estimator for the tail probability of the form
\[
c_\ast q_x(t) \exp ( - \lambda t) ,
\text{ for } t \in [T, \infty),
\]
for some choice of large threshold $T$, where $\lambda$ is the principal eigenvalue of the Dirichlet problem \eqref{eq: eigenvalue prob}
for some choice of large threshold $n$,
and $c_\ast$ is chosen so that the density estimator is continuous. The principal eigenvalue can be numerically computed from the Sturm-Liouville problem
\[
- \Big[ \exp \Big( 2 \int^{z}_{0} a(u) \ud u\Big) \,
\varphi^{\prime}(z) \Big]^{\prime}
= 2 \lambda \exp \Big( 2 \int^{z}_{0} a(u) \ud u \Big) \varphi(z) \,
\quad \text{ for } z \in [0,n] \,
\]
(see \cite{MR2170950}) with the same Dirichlet boundary conditions,
by use of either the variational method or the Liouville transform method.
In the variational method, the principal eigenvalue
is obtained by minimizing numerically the corresponding
Rayleigh quotient. The Liouville transform method turns
the Sturm-Liouville equation into a Schr\"odinger-type equation, which is then numerically solvable with discrete approximation. Both numerical
methods are well studied --- see \cite{MR1460546} and their references within.

\appendix

\section{Proofs} \label{sec: proofs}

\subsection{Proof of Proposition \ref{prop: p with BB3}}

Consider the non-negative $\probax$-supermartingale 
\begin{equation} \label{eq: def Z} 
Z \dfn \exp  \Big( - \int^{\cdot \wedge \tau_0}_{0} a(X_{u})  \ud W_{u}- \frac{1}{2} \int^{\cdot \wedge \tau_0}_{0} a^{2}(X_{u})  \ud u \Big).
\end{equation}
As follows from a modification of \cite[Exercise 5.5.38]{MR1121940} for the restricted state space $ [0, \infty)$, $Z$ is a $ \probax$-martingale. With $\qprob_x$ being the probability on $C(\Real_+; \Real)$ that makes the coordinate process $ X $ behave like a Brownian 
motion staring from $x$ and stopped when it reaches level zero, Girsanov's theorem implies that
\begin{equation} \label{eq: def Q}
\frac{\ud  \qprob_x}{\ud  \probax} \Big|_{\F_t} = Z_{t}.  
\end{equation}
for all $t \in \Real_+  $.  
Moreover, since $ X_{\tau_{0}} = 0 $, It\^o's formula (under $\qprob_x$) implies that, on the set $\{ \tau_{0} < \infty \}$, 
\[
- \int^{x}_{0} a(v)  \ud v = \int^{X_{\tau_{0}}}_{x} a(v)  \ud v = \int^{\tau_{0}}_{0} a(X_{u})  \ud X_{u} + \frac{1}{2} \int^{\tau_{0}}_{0} a^{\prime} (X_{u}) \ud u  . 
\]
Combining this with (\ref{eq: diffusion}) and (\ref{eq: gamma}), the stochastic exponential defined in 
(\ref{eq: def Z}) under $ \qprob_x$ satisfies 
\begin{equation} \label{eq: Z inv}
\begin{split}
\frac{1}{Z_{\tau_{0}}} &= \exp  \Big( \int^{\tau_{0}}_{0} a(X_{u})  
[ \ud X_{u} - a(X_{u}) \ud u]  + \frac{1}{2} \int^{\tau_{0}}_{0} a^{2}(X_{u})  \ud u \Big) \\
&= \exp  \Big( \int^{\tau_{0}}_{0} a(X_{u})  
\ud W^{\qprob_x}_{u}   + \frac{1}{2} \int^{\tau_{0}}_{0} a^{2}(X_{u})  \ud u \Big) \\
&= \exp  \Big( - \int^{x}_{0} a(v)  \ud v - \frac{1}{2} \int^{\tau_{0}}_{0} a^{\prime} (X_{u}) \ud u - \frac{1}{2} \int^{\tau_{0}}_{0} a^{2}(X_{u})  \ud u \Big) \\
& =  \exp  \Big( - \int^{x}_{0} a(v)  \ud v  - \int^{\tau_{0}}_{0} \gamma (X_{u})  \ud u \Big), 
 \end{split}
\end{equation}
on $\{ \tau_0 < \infty \}$, where $W^{\qprob_x}$ is a standard Brownian motion under $\qprob_x$. Therefore, (\ref{eq: def Q}) and (\ref{eq: Z inv}) imply
\begin{equation*}
\probax [\tau_{0} \le t] = \expec_{\qprob_x} \bra{ \frac{1}{Z_{t\wedge \tau_0}} \indic_{\{\tau_{0} \le t\}} } = 
\expec_{\qprob_x} \bra{ \exp  \Big( - \int^{x}_{0} a(v)  \ud v  - \int^{\tau_{0}}_{0} \gamma (X_{u})  \ud u  \Big) \indic_{\{\tau_{0} \le t\}} } \end{equation*}
for $(t,x) \in (\Real_+)^2$.

Note that the density function of $ \tau_{0} $ under $\qprob_x$ is given by $q_{x}$ in (\ref{eq: dens 0}). Using the regular conditional $\qprob_x$-expectation of $ \exp ( \int_{0}^{\tau_{0}} \gamma (X_{u})  \ud u) $, given $ \tau_{0} = t $, we can write
\begin{equation} \label{eq: pax Q}
\pax (t) = \probax[\tau_{0} \in \ud t ] = q_{x}(t) \exp  \Big (- \int^{x}_{0} a(v)  \ud v
 \Big) \expec_{\qprob_x} \Big[ \exp \Big( - \int^{\tau_{0}}_{0} \gamma(X_{u})  \ud u \Big)  \Big \vert \tau_{0} 
 = t\Big].  
\end{equation} 
Given $ \tau_{0} = t $, the regular conditional $\qprob_x$-distribution of $ (X_{t-s},  0 \le s \le t) $ is that of a 3-dimensional Bessel bridge from $0$ to $x$ over $[0,t]$ as a consequence of \cite[Proposition VI.3.10 and Proposition  VII.4.8]{RY99}. On the canonical space $ (C([0,1], \Real^{3}), \probbb) $ 
with coordinate process $ \beta $, the process $\set{ \lvert   (s / t) x \e_{1} + \sqrt{t} \beta_{s/t}\rvert,  0 \le s \le t}$ has the exact law of the aforementioned Bessel bridge. Therefore, 
\begin{equation*}
\begin{split}
& \expec_{\qprob_x} \Big[ \exp \Big( - \int^{\tau_{0}}_{0} \gamma(X_{u})  \ud u 
\Big)  \Big \vert \tau_{0} 
 = t\Big] = \expecbb\Big [ \exp  \Big( - \int^{t}_{0} \gamma \Big( \Big \lvert 
 \frac{s}{t} x \e_{1} + \sqrt{t} \beta_{s/t} \Big \rvert \Big)  \ud s  \Big)  \Big] \\
&=  \expecbb \left[ \exp \left(  - t\int_0^1 \gamma \left( \left|  u x \e_1 + \sqrt{t} \beta_u \right|
\right) \ud u \right) \right] ; \quad (t, x) \in (\Real_+)^2 .  
 \end{split}
\end{equation*}
Combining this with \eqref{eq: pax Q}, the proof of Proposition \ref{prop: p with BB3} is complete.

\subsection{Proof of Proposition \ref{prop: tight}}
The following technical result is the backbone of the proof.

\begin{lem}\label{lem: Lip I}
Suppose that Assumption \ref{ass: extra} holds, and define
\begin{equation} \label{eq: Iit}
\xi (t) \dfn \exp ( -t I (t)) := \exp 
\Big( -t \int_0^1 \gamma \left( 
\abs{  u x \e_1 + \sqrt{t} \beta_u}
\right) \ud u \Big), \quad t \in \Real_{+}.
\end{equation}
Then, we have
\[
|\xi (t) - \xi (s)|  \leq \Phi_T \abs{t-s} \text{ for all } s \in [0, T], \text{ and } t \in [0, T],   
\]
where $\expecbb \bra{|\Phi_T|^m} < \infty$ for all $T \in \Real_+$ and $m \in \Natural$.
\end{lem}

\begin{proof}
First, note that (\ref{eq: poly growth}) in 
Assumption \ref{ass: extra} implies that for every 
$ \kappa > 0$,   
\begin{equation*}
\begin{split}
\lvert a^\prime(v_1) \rvert \le \lvert a^\prime(0) \rvert 
+ \lvert a^\prime (v_1) - a^\prime (0) \rvert 
\le \lvert a^\prime(0) \rvert + c_1 v_1 + c_2 \kappa^n v_1
 ; \quad 0 \le v_1 \le \kappa ,  \qquad \\
\lvert a(\kappa) \rvert = \Big \lvert 
a (0) + \int^\kappa_0 a^\prime(u) 
\ud u \Big \rvert \le 
\lvert a (0) \rvert + 
\int^\kappa_0 \lvert a^\prime (u) \rvert \ud u 
\le \lvert a(0) \rvert + \lvert a^\prime (0) \rvert \kappa + c_1 \kappa^2
+ c_2 \kappa^{n+2}, \\
\lvert a(v_1) - a(v_2) \rvert \le \Big \lvert \int^{v_2}_{v_1} 
a^\prime (u) \ud u \Big \rvert \le 
\big (\lvert a' (0) \rvert + c_1 \kappa 
+c_2 \kappa^{n+1} \big) \lvert v_2 - v_1 \rvert  ; 
\quad 0 \le v_1 \le v_2 \le \kappa  . 
\end{split}
\end{equation*}
Using these inequalities, we obtain estimates for 
$ \gamma$ 
for every $ \kappa > 0$ and 
every $ 0 \le v_1 \le v_2 \le \kappa$,  
\begin{equation*} 
\abs{\gamma(v_1)} = \abs{\frac{a^2(v_1) + a^\prime(v_1)}{2}}
\le \frac{1}{2} \big( \lvert a(0) \rvert + \lvert a^\prime (0) \rvert \kappa + c_1 \kappa^2 + c_2 \kappa^{n+2}
\big)^2 + \frac{1}{2} \big( \lvert a^\prime(0) \rvert + c_1 \kappa +
c_2 \kappa^{n+1} \big) =: \varphi_1(\kappa),
\end{equation*}
\begin{equation} \label{eq: varphi 2}
\begin{split}
&\abs{ \gamma(v_1) - \gamma(v_2)}
\le \frac{1}{2} \abs{a^2(v_1) - a^2(v_2)} +
\frac{1}{2}\abs{a^\prime (v_1) - a^\prime(v_2)}  \\
& \le \frac{1}{2}\abs{a(v_1)+a(v_2)}\abs{a(v_1) - a(v_2)}  + 
\frac{1}{2} (c_1 + c_2 \kappa^n) \abs{v_1-v_2} 
\\
&  \le
\Big[ \big(
\lvert a(0) \rvert + \lvert a^\prime (0) \rvert \kappa + c_1 \kappa^2 + c_2 \kappa^{n+2} \big)\big( \lvert a^\prime(0) \rvert + c_1 \kappa +
c_2 \kappa^{n+1} \big) + \frac{1}{2} (c_1 + c_2 \kappa^n)
\Big] \abs{v_1 - v_2}
\\
& \qquad \qquad \qquad   
=: \varphi_2(\kappa) \abs{v_1 - v_2} , 
\end{split}
\end{equation}
where $\Real_+ \ni \kappa \mapsto \varphi_j (\kappa)$, $ j = 1, 2$, are polynomial functions of $ \kappa \in \Real_+$ and do not depend on 
$ v_1 $, $ v_2$.

Fix $T \in \Real_+$. For $s \in [0, T]$, $t \in [0, T]$ and $u \in [0,1]$, consider the random variables $ \kappa = \sqrt{T} \max_{0 \le u \le 1} \lvert \beta_u \rvert  + x$, $v (s, u) = 
\lvert ux \e_1 + \sqrt{s} \beta_u \rvert$, $v (t, u) = 
\lvert ux \e_1 + \sqrt{t} \beta_u \rvert$. Using the estimates established before, we obtain estimates for $ I (t)$ in (\ref{eq: Iit}): 
\begin{equation*} 
\abs{I (t)} = 
\Big \lvert 
\int^1_0 \gamma ( \lvert u x e_{1}+ \sqrt{t} \beta_u \rvert 
)\ud u \Big \rvert 
\le \int^1_0 \big \lvert 
\gamma ( \lvert 
u x e_{1}+ \sqrt{t} \beta_u \rvert ) \big \rvert \ud u 
= \int^1_0 \big \lvert \gamma (v(t, u)) \rvert \ud u 
\le \varphi_1 ( \kappa) , 
\end{equation*}
\begin{equation}\label{eq: ineq for second xi}
\begin{split}
s\abs{I^{1}(t)- I^{1}(s)} &\le 
s \int^1_0 \lvert \gamma (v (t, u)) 
- \gamma(v (s, u)) \rvert \ud u 
\le s \int^1_0 
\lvert v (t, u) - v (s, u) \rvert \varphi_2 (\kappa) \ud u \\
& \le s (\sqrt{t} - \sqrt{s}) 
\int^1_0 \lvert 
\beta_u \rvert \varphi_2(\kappa ) \ud u 
\le \frac{s (t-s)}{\sqrt{t} + \sqrt{s}} 
\, \frac{\kappa}{\sqrt{T}}
\, \varphi_2 (\kappa ) \\
&\le \frac{s \kappa \varphi_2(\kappa )}
{2 \sqrt{sT}} \, (t-s)
\le \frac{\kappa \varphi_2(\kappa )}{2}
\, (t-s)  ,   
\end{split}
\end{equation}
where we have used (\ref{eq: varphi 2}) in the second inequality, 
since 
$ 0 \le v(s, u) \le \kappa $, 
$ 0 \le v(t, u) \le \kappa $, 
used 
$ \lvert v (t, u) - v (s, u) \rvert \le 
(\sqrt{t} - \sqrt{s}) \lvert \beta_u \rvert$ 
in the third inequality, 
and $ \max_{0 \le u \le 1} \lvert \beta_u \rvert 
\leq \kappa /\sqrt{T}$ in the fourth inequality 
for $ 0 < s < t \le T$.    

Finally, since $ \gamma $ is bounded from below 
by Assumption \ref{ass: extra}, so is 
$ I $ in (\ref{eq: Iit}), that is, 
$ I^1(t) \ge \inf_{z \in \Real_+} \gamma(z) > - \infty$ for every $ t \ge 0$. 
With this observation, because of monotonicity and differentiability of exponential function, 
we obtain for 
$ 0 \le s < t \le T  $, 
\[
\begin{split}
& \qquad \qquad 
 \abs{\xi (t) - \xi (s)} =
\abs{e^{-t I (t)} - e^{-s I (s)}} 
\le (e^{-T \inf \gamma} \vee 1) \abs{t I (t) - s I (s)} \\
& = c_3 \lvert (t-s) I (t) + s(I (t) - I (s)) \rvert 
\le c_3 \abs{t-s}\abs{I(t)} 
+ c_3 s \abs{I (t) - I (s)}, 
\end{split}
\]
where $ c_3 := \exp ( - T \inf \gamma ) \vee 1 < \infty$.  
Combining this with the estimates for  
$ \lvert I (t) \rvert$ and 
$ s\lvert I (t) - I (s) \rvert $ in 
(\ref{eq: ineq for second xi}), we obtain, for $ 0 \le s < t \le T $, 
\[
|\xi (t) - \xi (s)| \le c_3 
\Big( 
\varphi_1(\kappa) + \frac{\kappa \varphi_2(\kappa)}{2} 
\abs{t-s} \Big) =: \varphi_3(\kappa) \abs{t-s} ,   
\]
where $ \varphi_1$ and $ \varphi_2$ are 
defined in (\ref{eq: varphi 2}), and hence 
$ \varphi_3$ can be written as 
a polynomial function of $ \kappa$ whose coefficients 
do not depend on $ s$ nor on $ t$ but on $ T$. Letting $\Phi_T$ be $\varphi_3(\kappa)$, and  noting that all positive integer moments 
of maximum of standard $ 3$-dimensional Bessel Bridge are
finite (\cite[Corollary 7]{MR1701890}), we conclude the proof of Lemma \ref{lem: Lip I}.
\end{proof}

Let us define $\nu(t):= \expecbb \bra{\xi (t)}$ for $ 0 \le t \le T$. It follows from Lemma \ref{lem: Lip I} that $ \xi$ is locally Lipschitz continuous and moreover, 
\[
\expecbb \bra{\left| \xi (t) - \xi (s) - (\nu(t) - \nu(s)) \right|^{2} }
\le c_{4} \abs{t-s}^{2} ; \quad  0 \le s < t \le T  . 
\]
Since the random paths $ \{\widehat{\beta}^i, i = 1, \ldots , N\}$ are independent and identically distributed, for any $ s, t \le T $ we obtain 
\begin{equation} \label{eq: estimate on module}
\begin{split}
&\expecbb \big \lvert 
\eta^{N}(t) - \eta^{N}(s) \big \rvert ^{2} \\
&= \big ( q_{x}(t) \big)^{2} 
\exp \Big( - 2\int_0^x a(v) \ud v \Big)
 \expecbb \bra{ \big \lvert \xi (t) - \xi (s) - (\nu(t) - \nu(s)) \big \rvert^{2} } \\
& \le \big( \max_{0 \le t \le T}q_{x}(t) \big)^{2}
\exp \Big( - 2\int_0^x a(v) \ud v \Big)
 c_{4} \abs{t-s}^{2} =: c_{5} \abs{t-s}^{2},
\end{split}
\end{equation}
where the constant $ c_{5}$ 
depends on $ T, x $ but not on $N, s, t$. This inequality is a sufficient condition for the tightness of the sequence $\set{\eta^{N} \such N \in \Natural}$ of continuous stochastic processes starting at $0$  in $ C(\Real_+, \Real)$ --- see \cite[Problem 2.4.12]{MR1121940}. By the usual multi-dimensional central limit theorem, for each $n\ge 1$, $0\le t_{1} < \cdots < t_{n}< \infty$ the sequence $\set{ \pare{\eta^{N}(t_{1}), \eta^{N}(t_{2}), \ldots , \eta^{N}(t_{n})}  \such N \in \Natural}$ of random vectors converges in distribution to a Gaussian random vector with mean zero and and variance-covariance matrix $\pare{\Gamma(t_i, t_j)}_{1 \leq i, j \leq n}$, where
\[
\Gamma(s,t)  =
\exp \Big( -2\int^x_0 a(v) \ud v\Big)  \covbb [\exp (-sI(s)), 
\exp (-t I(t))], \quad (s,t) \in \Real_+^2.  
\]
Therefore,  we conclude that the tight sequence $\set{\eta^{N} \such N \in \Natural}$ converges weakly to a continuous Gaussian process with mean zero and continuous covariance function $\Gamma$.

\subsection{Proof of Proposition \ref{prop: conv in prob}}

Define that Gaussian tail function $ \bar \Phi $ via
\[
\bar \Phi (z) = \int^{\infty}_{z} \frac{e^{-y^{2}/2}}{\sqrt{2\pi}} \ud y, \text{ for } z \in \Real.
\]
Furthermore, for fixed $T \in \Real_+$, define 
the modulus of continuity in $ \mathbb L^2$: 
\begin{equation} \label{eq: varphi}
\psi_T (h) := \max_{(s,t) \in [0,T]^2 ,  |t-s| \le h} \left[ \expecbb ( \eta (t) - \eta(s)) ^{2} \right]^{1/2},
\text{ for }  h \in [0, T]. 
\end{equation} 
It follows from (\ref{eq: estimate on module}) that $ \psi_T (h) \le \sqrt{c_{5}} h$ for $h \in \Real_+$; therefore, $ \int_1^\infty \psi_T (e^{-y^2}) \ud y < \infty$. We recall the following Fernique's inequality 
for Gaussian processes, which we shall use. 

\begin{lem}[Fernique's inequality --- see (2.2) of \cite{MR791269}] \label{lem: Berman}
If the function $\psi_T$ in (\ref{eq: varphi}) satisfies 
$ \int^\infty_1 \psi_T ( e^{-y^2}) \ud y$, 
then for that fixed $ T>0 $ and any integer $m \ge 2 $, 
\[
\probbb \bra{ \max_{0 \le t \le T} \abs{ \eta(t) } > C_{1}(T, m) z
} \le C_{2}(m) \bar \Phi(z), \quad 
\text{ for all } z > (1+ 4 \log m )^{1/2} , 
\]
where $ C_{1}(T, m):= \max_{(s,t)\in [0,T]^{2}} 
\Gamma(s,t)^{1/2} + (2+\sqrt{2}) \int_{1}^{\infty}
\psi_T (T m^{-y^{2}}) \ud y$, 
and $ C_{2}:= 5 m^2 \sqrt{2\pi}/2$. 
\end{lem}

The weak convergence of $\set{\eta^N \such N \in \Natural}$ to $\eta$ and the invariance principle for the maximum function imply that the sequence of normalized maxima $\pare{ \sqrt{N}  \max_{t \in [0, T]} \abs{\phx (t)- \px(t)}}_{N \in  \Natural}$ converges weakly to $\max_{0 \le t \le T} \eta (t)$, where $\eta$ is the limiting Gaussian process, as $N$ goes to infinity. Since the law of the last random variable does not charge $\infty$, we conclude that the family $\set{ \sqrt{N}  \max_{t \in [0, T]} \abs{\phx (t)- \px(t)} \such N \in  \Natural}$ is bounded in probability.

\subsection{Proof of Proposition \ref{lm: bounds lambda}}

For the lower bound in (\ref{eq: estimate lgx limit}) let us observe
\[
\log \expecbb \left[ \exp \left( - t \int^{1}_{0} \gamma \left( \left|  u x \e_1 + \sqrt{t} \beta_u \right|
\right) \ud u \right) \right]  \le 
-t \inf_{z\in \Real_{+}} \gamma (z) ; \quad t \in \Real_+,  
\]
by Assumption \ref{ass: extra}. Therefore,
\[
\inf_{z \in \Real_{+}} \gamma (z) \le 
\inf_{t \in \Real_{+}} \left[ - \frac{1}{t}
\log \expecbb \left[ \exp \left( - t \int^{1}_{0} \gamma \left( \left|  u x \e_1 + \sqrt{t} \beta_u \right|
\right) \ud u \right) \right] \right]=  
\inf_{t \in \Real_{+}} 
\lgx (t) . 
\]
For the upper bound in (\ref{eq: estimate lgx limit}), 
for a fixed $\kappa > 0$ let us consider $ A_\kappa \dfn \set{ \sqrt{t} \max_{0 \le u \le 1} \lvert \beta_{u} \rvert \le \kappa}$. On $A_\kappa$, $ \lvert u x \e_{1} + \sqrt{t} \beta_{u} \rvert \le \kappa + x $ holds for 
$ 0 \le u \le 1$; hence, 
\[
 \int^{1}_{0} \gamma \left( \left|  u x \e_1 + \sqrt{t} \beta_u \right|
\right) \ud u \le \max_{0\le z \le \kappa + x} \gamma (z) = m(\kappa + x), 
\]
where $ m(w) \dfn \max_{0 \le z \le w} \gamma (z) $ for $ w > 0$.  
It follows that, for $ t \in \Real_+$, 
\begin{equation*}
\begin{split}
\expecbb \left[ \exp \left( - t \int^{1}_{0} \gamma \left( \left|  u x \e_1 + \sqrt{t} \beta_u \right|
\right) \ud u \right) \right] 
&\ge \expecbb \left[ \exp \left( - t \int^{1}_{0} \gamma \left( \left|  u x \e_1 + \sqrt{t} \beta_u \right|
\right) \ud u \right) \indic_{A_\kappa} \right ]  \\
&\ge \exp \left( - t m ( \kappa + x) \right) \probbb
\Big [ \sqrt{t}\max_{0\le u \le 1} \lvert \beta_{u} \rvert \le \kappa \Big] ,  
\end{split}
\end{equation*}
and hence 
\begin{equation} \label{eq: upper pre-bound} 
\lgx(t) \leq m(\kappa + x) - \frac{1}{t} \log 
\probbb \bra{ \max_{0 \le u \le 1} \lvert \beta_u \rvert 
\le \kappa t^{-1/2} }. 
\end{equation}
The distribution for the maximum of the absolute value of the standard 3-dimensional Bessel bridge $|\beta|$ is known --- see, for example, \cite[(5)]{MR1701890}. More precisely, we have
\begin{equation} \label{eq: bes br dist}
\probbb \bra{ \max_{0 \le u \le 1} \lvert \beta_u \rvert 
\le \kappa t^{-1/2} } 
= \frac{2}{\kappa^3}\sqrt{\frac{2t}{\pi}} 
\sum_{n=1}^\infty \frac{n \pi}{ J^2_{3/2} (n\pi)} 
\exp \Big( - \frac{\pi^2 n^2 }{2 \kappa^2} t\Big), \text{ for } 
 \kappa > 0 \text{ and }   t \in \Real_+,   
\end{equation}
where $ J_{3/2}$ is the Bessel function of 
index $ 3/2$. In particular,
\[
\lim_{t \to \infty} \pare{- \frac{1}{t} \probbb \log \bra{ \max_{0 \le u \le 1} \lvert \beta_u \rvert \le \kappa t^{-1/2} }} = \frac{\pi^2}{2 \kappa^2},
\]
as only the first term in the summand in the series in \eqref{eq: bes br dist} will play a role in the limit. Combining the last limiting relationship with inequality \eqref{eq: upper pre-bound}, we obtain
\[
\limsup_{t \to \infty} \lgx (t) \leq m(\kappa + x) + \frac{\pi^2}{2 \kappa^2}
\]
Upon minimizing the right-hand side of the last inequality,
the upper bound in  (\ref{eq: estimate lgx limit}) is obtained.

\smallskip

Finally, to verify (\ref{eq: lgx at zero}), observe that $\gamma$ being bounded from below implies that the random variables $ \exp (- t \int^1_0 \gamma( \lvert u x \e_1 +\sqrt{t} \beta_u \rvert ) \ud u )$ 
are uniformly bounded for small $t \in \Real_+$. Then, de L'H\^ospital's rule and the bounded convergence 
theorem give
\begin{equation*}
\lim_{t\downarrow 0} \lgx(t) = 
(-1) \lim_{t \downarrow 0} 
\frac{\pare{\partial  / \partial t} 
\expecbb [\exp ( - t \int^1_0 \gamma( \lvert u x \e_1 
+\sqrt{t} \beta_u \rvert ) \ud u )]}
{\expecbb [\exp ( - t \int^1_0 \gamma( \lvert u x \e_1 
+\sqrt{t} \beta_u \rvert ) \ud u )]} =
\int^1_0 \gamma( u x )\ud u 
= \frac{1}{x} \int^x_0 \gamma (u) \ud u. 
\end{equation*}

\bibliographystyle{siam}
\bibliography{first_passage_time_density}

\def\cprime{$'$}
\begin{thebibliography}{10}

\bibitem{MR1460546}
{\em Spectral theory and computational methods of {S}turm-{L}iouville
  problems}, vol.~191 of Lecture Notes in Pure and Applied Mathematics, New
  York, 1997, Marcel Dekker Inc.
\newblock Edited by Don Hinton and Philip W. Schaefer.

\bibitem{MR2591904}
{\sc L.~Alili and P.~Patie}, {\em Boundary-crossing identities for diffusions
  having the time-inversion property}, J. Theoret. Probab., 23 (2010),
  pp.~65--84.

\bibitem{MR1849251}
{\sc P.~Baldi, L.~Caramellino, and M.~G. Iovino}, {\em Pricing general barrier
  options: a numerical approach using sharp large deviations}, Math. Finance, 9
  (1999), pp.~293--322.

\bibitem{MR791269}
{\sc S.~M. Berman}, {\em An asymptotic bound for the tail of the distribution
  of the maximum of a {G}aussian process}, Ann. Inst. H. Poincar\'e Probab.
  Statist., 21 (1985), pp.~47--57.

\bibitem{MR776891}
{\sc J.~Durbin}, {\em The first-passage density of a continuous {G}aussian
  process to a general boundary}, J. Appl. Probab., 22 (1985), pp.~99--122.

\bibitem{MR0068701}
{\sc J.~Elliott}, {\em Eigenfunction expansions associated with singular
  differential operators}, Trans. Amer. Math. Soc., 78 (1955), pp.~406--425.

\bibitem{MR1725406}
{\sc K.~D. Elworthy, X.-M. Li, and M.~Yor}, {\em The importance of strictly
  local martingales; applications to radial {O}rnstein-{U}hlenbeck processes},
  Probab. Theory Related Fields, {\bf 115} (1999), pp.~325--355.

\bibitem{MR1999614}
{\sc P.~Glasserman}, {\em Monte {C}arlo methods in financial engineering},
  vol.~53 of Applications of Mathematics (New York), Springer-Verlag, New York,
  2004.
\newblock Stochastic Modelling and Applied Probability.

\bibitem{MR1757112}
{\sc E.~Gobet}, {\em Weak approximation of killed diffusion using {E}uler
  schemes}, Stochastic Process. Appl., 87 (2000), pp.~167--197.

\bibitem{GJY03}
{\sc A.~G\"oing-Jaeschke and M.~Yor}, {\em A clarification note about hitting
  times densities for ornstein-uhlenbeck processes}, Finance and Stochastics, 7
  (2003), pp.~413--415.

\bibitem{MR1121940}
{\sc I.~Karatzas and S.~E. Shreve}, {\em Brownian motion and stochastic
  calculus}, vol.~113 of Graduate Texts in Mathematics, Springer-Verlag, New
  York, second~ed., 1991.

\bibitem{MR576891}
{\sc J.~T. Kent}, {\em Eigenvalue expansions for diffusion hitting times}, Z.
  Wahrsch. Verw. Gebiete, 52 (1980), pp.~309--319.

\bibitem{MR1214374}
{\sc P.~E. Kloeden and E.~Platen}, {\em Numerical solution of stochastic
  differential equations}, vol.~23 of Applications of Mathematics (New York),
  Springer-Verlag, Berlin, 1992.

\bibitem{MR2061247}
{\sc G.~N. Milstein and M.~V. Tretyakov}, {\em Evaluation of conditional
  {W}iener integrals by numerical integration of stochastic differential
  equations}, J. Comput. Phys., 197 (2004), pp.~275--298.

\bibitem{MR2069903}
\leavevmode\vrule height 2pt depth -1.6pt width 23pt, {\em Stochastic numerics
  for mathematical physics}, Scientific Computation, Springer-Verlag, Berlin,
  2004.

\bibitem{MR1984752}
{\sc G.~Peskir}, {\em On integral equations arising in the first-passage
  problem for {B}rownian motion}, J. Integral Equations Appl., 14 (2002),
  pp.~397--423.

\bibitem{MR1326606}
{\sc R.~G. Pinsky}, {\em Positive harmonic functions and diffusion}, vol.~45 of
  Cambridge Studies in Advanced Mathematics, Cambridge University Press,
  Cambridge, 1995.

\bibitem{MR1701890}
{\sc J.~Pitman and M.~Yor}, {\em The law of the maximum of a {B}essel bridge},
  Electron. J. Probab., 4 (1999), pp.~no. 15, 35 pp.1--35 (electronic).

\bibitem{RY99}
{\sc D.~Revuz and M.~Yor}, {\em Continuous {M}artingales and {B}rownian
  {M}otion}, Volume {\bf 293} of Grundlehren der Mathematischen Wissenschaften
  $[$Fundamental Principles of Mathematical Sciences$]$, Springer-Verlag,
  Berlin, third~ed., 1999.

\bibitem{MR2170950}
{\sc A.~Zettl}, {\em Sturm-{L}iouville theory}, vol.~121 of Mathematical
  Surveys and Monographs, American Mathematical Society, Providence, RI, 2005.

\end{thebibliography}

\end{document}